\documentclass[11pt,reqno]{amsart}

\usepackage[english]{babel}

\usepackage{amsfonts}
\usepackage{amsmath}
\usepackage{amssymb}
\usepackage{amsthm}
\usepackage{tikz}

\usepackage[colorlinks,link color=blue, cite color=blue,pagebackref,pdftex]{hyperref}

\renewcommand\emptyset{\varnothing}
\renewcommand\phi{\varphi}
               
\newcommand\Z{\mathbb{Z}}

\newcommand\R{\mathbb{R}}

\DeclareMathOperator*{\relint}{relint}

\DeclareMathOperator*{\conv}{conv}

\DeclareMathOperator*{\Ehr}{E}

\DeclareMathOperator*{\exc}{exc}

\DeclareMathOperator*{\des}{des}

\DeclareMathOperator*{\bad}{bad}

\newtheorem{thm}{Theorem}[section]
\newtheorem*{thm*}{Theorem}
\newtheorem{cor}[thm]{Corollary}
\newtheorem{lem}[thm]{Lemma}
\newtheorem{prop}[thm]{Proposition}

\newtheorem{conj}{Conjecture}

\theoremstyle{definition}

\newcommand{\A}{\mathcal{A}}

\title{The Eulerian transformation}

\author{Petter Br\"and\'en}
\address{Department of Mathematics, %
KTH Royal Institute of Technology, \mbox{SE-100 44} Stockholm,
Sweden}
\email{pbranden@math.kth.se}

\author{Katharina Jochemko}
\address{Department of Mathematics, %
KTH Royal Institute of Technology, \mbox{SE-100 44} Stockholm,
Sweden}
\email{jochemko@kth.se}


\date{\today}

\begin{document}

\maketitle

\begin{abstract}
Eulerian polynomials are fundamental in combinatorics and algebra. In this paper we study the linear transformation $\A : \R[t] \to \R[t]$ defined by $\A(t^n) = A_n(t)$,  where $A_n(t)$ denotes the $n$-th Eulerian polynomial. We give combinatorial, topological and Ehrhart theoretic interpretations of  the operator $\A$, and investigate questions of unimodality and real-rootedness. In particular, we disprove a conjecture by Brenti (1989) concerning the preservation of real zeros, and generalize and strengthen recent results of Haglund and Zhang (2019) on binomial Eulerian polynomials.
\end{abstract}

\section{Introduction}
The \emph{Eulerian numbers}, $\{A_{nk}\}_{0\leq k \leq n}$, are among the most studied numbers in combinatorics and algebra, see the extensive survey~\cite{Petersen}. The \emph{Eulerian polynomials} may be defined by 
$$
A_n(t)= A_{n0}+ A_{n1}t + \cdots + A_{nn}t^n,
$$
where $A_{nk}$ is the number of permutations $\sigma$ in the symmetric group on $\{1,2,\ldots, n\}$ for which $\sigma(i) \geq i$ for exactly $k$ numbers $i$. 
\begin{align*}
A_0(t) &= 1, & A_3(t) &= t+4t^2+t^3, \\
A_1(t) &= t, & A_4(t) &= t+11t^2+11t^3+t^4, \\
A_2(t) &= t+t^2, & A_5(t) &= t+26t^2+66t^3+26t^4+ t^5.
\end{align*}

Already Frobenius~\cite{Frobenius} knew that for all $n \geq 0$, all the zeros of $A_n(t)$ are real. This is not an isolated phenomenon. Many polynomials in combinatorics and algebra are known or conjectured to be unimodal, log-concave or real-rooted. Several techniques are now available for proving such results and conjectures, see~\cite{PetterUnimodality,BrentiSurvey,StanleySurvey}. One approach is to study linear transformations that preserve real-rootedness properties. This topic goes back to the works of Jensen, P\'olya, Schur and Szeg\H{o}, and has recently regained popularity and importance, see \cite{Borcea-Branden}, and the references therein. In 1989, Brenti~\cite{BrentiUnimodal} conjectured that the linear operator $\A : \R[t] \to \R[t]$ defined by $\A(t^n) = A_n(t)$ preserves the property of having only real and non-positive zeros. Although we disprove this conjecture here in its full generality (Proposition~\ref{prop:disproveconj}),  we go on to study more subtle properties of the operator $\A$. 

Of recent particular interest are binomial Eulerian polynomials, a variant of Eulerian polynomials defined by
\[
 \tilde{A}_n (t)=\sum _{i=0}^n{n\choose i}A_i (t)=\A ((t+1)^n)\, .
\]
They appear in work by Postnikov \emph{et al.}~\cite{Postnikovstellohedron} as the $h$-polynomial of the stellohedron. Binomial Eulerian polynomials share fundamental properties with (ordinary) Eulerian polynomials: Postnikov \emph{et al.}~\cite{Postnikovstellohedron}, and more recently Shareshian and Wachs~\cite{Wachs} showed that $\tilde A _n (t)$ is $\gamma$-positive and provided different combinatorial interpretations of the $\gamma$-coefficients. In particular, the coefficients of the binomial Eulerian polynomial $ \tilde{A}_n (t)$  form a symmetric and unimodal sequence for each $n\geq 0$. Recently, Haglund and Zhang~\cite{HaglundZhang}  proved that the binomial Eulerian polynomials are real-rooted, affirming a conjecture by Ma \emph{et al.}~\cite{ma2017recurrence}. Indeed, they proved real-rootedness of the binomial Eulerian polynomials for $r$-colored permutations introduced by Athanasiadis~\cite{Athanasiadis}, which may be seen to be equal to $\tilde{A}_{n,r}(t)=\A ((rt+1)^n)$. 

We conjecture that $\A(f)$ is real-rooted whenever $f$ is a nonnegative linear combination of the polynomials $\{t^i(t+1)^{n-i}\}_{i=0,1,\ldots,n}$ (Conjecture~\ref{conj:newconj}). This would imply that if $f$ is real-rooted and all its zeros are contained in the interval $[-1,0]$, then $\A(f)$ is real-rooted. In support of this conjecture we show in Section~\ref{sec:combint} that $\A (f)$ has alternatingly increasing and thus unimodal coefficients whenever $f$ is a nonnegative linear combination of the polynomials $\{t^i(t+1)^{n-i}\}_{i=0,1,\ldots,n}$ (Theorem~\ref{thm:alternating}).  The proof uses a combinatorial interpretation in terms of the excedance statistics on decorated permutations (Proposition~\ref{prop:combint}). In Section~\ref{sec:combint}  we also give a new proof for the $\gamma$-positivity of $\tilde A _n (t)$ (Theorem~\ref{gammaexp}) providing an interpretation for the $\gamma$-coefficients that is different from the previous ones given in~\cite{Postnikovstellohedron,Wachs}. In Section~\ref{sec:real-rooted} we prove that $\A ((t+q)^n)$ is real-rooted for all real numbers $0\leq q \leq 1$ (Theorem~\ref{thm:main1}). In fact, we prove that $\A ((t+q)^n)$ interlaces $\A ((t+p)^n)$ for all $0\leq p\leq q\leq 1$ (Theorem~\ref{thm:interlacing}), and that $\A ((t+p)^n)$ has an interlacing symmetric decomposition (Corollary~\ref{cor:interlacingsym}). This refines and strengthens the aforementioned results by Haglund and Zhang~\cite{HaglundZhang}. 

We conclude our discussion by giving topological and Ehrhart theoretic interpretations of the linear operator $\A$. Given a simplicial complex $\Delta$ with $f$-polynomial $f_\Delta (t)$ we construct in Section~\ref{sec:topolint} a simplicial complex $\Delta '$ such that $\A (f_\Delta)$ is equal to the $h$-polynomial of $\Delta '$ (Theorem~\ref{thm:topoint}). In Section~\ref{Ehrhart} we give an interpretation for $\A ((\theta _1 t+1)\cdots (\theta _n t+1))$ as the $h^\ast$-polynomial of a lattice polytope whenever $\theta _1,\ldots, \theta _n$ are nonnegative integers (Theorem~\ref{thm:hstarinterpretation}). In particular, the considered lattice polytopes have unimodal $h^\ast$-polynomial, a property of particular interest in Ehrhart theory, see the survey~\cite{braun2016unimodality}.

\section{Counterexamples to Brenti's conjecture}
Brenti~\cite{BrentiUnimodal} conjectured that if $f$ is a polynomial with nonnegative coefficients and only real zeros, then $\A(f)$ is a polynomial that has nonnegative coefficients and only real zeros. To see that this conjecture is false, consider the polynomials 
\begin{equation}\label{eulerseries}
p_n(t) = \A\left(\left( 1+\frac {xt}  n \right)^n\right)= \sum_{k=0}^n \binom n k A_k(t) (x/n)^k =: \sum_{k=0}^n a_{nk} \frac {t^k}{k!} 
\end{equation}
where $x>0$ and $n >0$. Notice that $a_{n0}=1$ and $a_{n1} = (1+x/n)^n-1$. It follows from the Newton inequalities~\cite[Lemma 1.1]{PetterUnimodality} that if 
$\sum_{k=0}^n b_k t^k /k!$ is a real-rooted  polynomial with nonnegative coefficients and $b_0=1$, then
$$
b_k^2 \geq b_{k-1}b_{k+1}, \ \ \ \mbox{ for all } k \geq 1, 
$$
from which it follows that $b_k \leq b_1^k$ for all $k$. Suppose now that $p_n(t)$ is real-rooted for all $n$ in some infinite set $\mathcal{S}$ of positive integers. Then $a_{nk} \leq a_{n1}^k\leq (e^x-1)^k$ for all $k \geq 0$ and $n \in \mathcal{S}$. Notice that 
$$
\lim_{n \to \infty} n^{-k}\binom n k = \frac 1 {k!}. 
$$
By dominated convergence it follows that the series \eqref{eulerseries} converges to an entire function (in $t$)
\begin{equation}\label{eulerexp}
\sum_{k=0}^\infty A_k(t) \frac {x^k}{k!}. 
\end{equation}
However, the series \eqref{eulerexp} has the form 
\begin{equation}\label{esexpr}
\frac {1-t}{1-t\exp(x-tx)},
\end{equation}
see~\cite[Proposition 1.4.5]{EC1}. Moreover, the denominator of \eqref{esexpr} is zero if and only if $(tx)e^{-tx}=xe^{-x}$. This equation (in $t$) has two solutions counting multiplicities. Hence \eqref{esexpr}  has a pole, and the series does not define an entire function, which is a contradiction. We conclude 
\begin{prop}\label{prop:disproveconj}
For each $x>0$, there exists an integer $N$ such that the polynomial (in $t$) 
 $$\A\left(\left( 1+\frac {xt}  n \right)^n\right)$$
 fails to be real-rooted for all $n \geq N$. 
\end{prop}
A concrete counterexample is the polynomial $\A(\left( 1+ {t}/5 \right)^5)$, which is seen to have two non-real zeros which are approximately $-1.79\pm 0.56i$. 

We make the following conjecture, and we will provide some partial evidence in the forthcoming sections. 

\begin{conj}\label{conj:newconj}
If $f= \sum_{k=0}^n a_k t^k (1+t)^{n-k}$, where $a_k \geq 0$ for all $0\leq k \leq n$, then $\A(f)$ is real-rooted. 
\end{conj}

\section{Combinatorial interpretation and unimodality}\label{sec:combint}
If $\sigma$ is a permutation in the symmetric group $\mathfrak{S}_n$ on $[n]:=\{1,2,\ldots, n\}$, then $\exc(\sigma)= |\{ i : \sigma(i) >i\}|$ is the number of \emph{excedances} in $\sigma$, and $F(\sigma)= \{i : \sigma(i) =i\}$ is the set of \emph{fixed points} of $\sigma$.   Recall that the Eulerian polynomials may be defined as 
$$
A_n(t)= \sum_{\sigma \in \mathfrak{S}_n} t^{\exc(\sigma)} t^{|F(\sigma)|}.
$$

Postnikov~\cite{postnikov2006total} considered the set $\mathfrak{S}_n^2$ of \emph{decorated permutations}, i.e., the set of all permutations of $[n]$ where the fixed points are colored by either $0$ or $1$. 
The \emph{excedance} statistic on $\mathfrak{S}_n^2$  is defined by 
$$
\exc(\sigma)= |\{ i : \sigma(i) >i\}|+ |F_1(\sigma)|,  
$$
where $F_c(\sigma)= \{ i : \sigma(i) = i \mbox{ and $i$ has color } c\}$. It follows that 
\begin{prop}\label{prop:combint}
For real numbers $\theta _1,\ldots, \theta _n$, 
\begin{equation}\label{eq:combint1}
\A\left((t+\theta_1)\cdots(t+\theta_n)\right)= \sum_{k=0}^n e_{n-k}(\theta_1,\ldots, \theta_n)A_k(t)= \sum_{\sigma \in \mathfrak{S}_n^2} t^{\exc(\sigma)} \prod_{i \in F_0(\sigma)}\theta_i \, ,
\end{equation}
where 
$$
e_k(\theta_1,\ldots, \theta_n)= \sum_{1 \leq i_1 <i_2 <\cdots < i_k \leq n} \theta_{i_1} \theta_{i_2}\cdots \theta_{i_k}, 
$$
is the $k$-th elementary symmetric polynomial in $\theta_1,\ldots, \theta_n$. 
\end{prop}
Similarly if $d_n(t)= \sum_\sigma t^{\exc(\sigma)}$, where the sum is over all derangements in $\mathfrak{S}_n$, i.e.,  the permutations in $\mathfrak{S}_n$ with no fixed points, then 
\begin{equation}\label{dn}
\A\left((t+\theta_1)\cdots(t+\theta_n)\right) = \sum_{k=0}^n e_{n-k}(t+\theta_1,\ldots, t+\theta_n)d_k(t). 
\end{equation}
Next, we will see that $\A\left((t+\theta_1)\cdots(t+\theta_n)\right)$ is \emph{unimodal} for all real numbers $0\leq \theta _1, \ldots, \theta _n \leq 1$, i.e., if $c_0,c_1,\ldots, c_n$ is the sequence of coefficients then
\[
c_0\leq c_1\leq \cdots \leq c_k \geq \cdots \geq c_n
\]
for some $0\leq k\leq n$. In fact, we will prove that the coefficients satisfy the \emph{alternatingly increasing property}
\[
0\leq c_0\leq c_n \leq c_1\leq c_{n-1}\leq \cdots \leq c_{\lfloor \frac{n+1}{2}\rfloor } \, .
\]

Let $f(t)$ be a polynomial of degree at most $n$, and let $\mathcal{I}_n(f)(t) := t^n f(1/t)$. We may write $f$ as 
\begin{equation}\label{staple}
f(t)= a_n(f)+ t\cdot b_n(f),
\end{equation}
where 
$$
a_n(f)= \frac { f-t\mathcal{I}_n(f)}{1-t} \ \ \ \ \mbox{ and } \ \ \ \ b_n(f)= \frac { \mathcal{I}_n(f)-f}{1-t}. 
$$
Note that $\mathcal{I}_n(a_n(f))=a_n(f)$ and $\mathcal{I}_{n-1}(b_n(f))=b_n(f)$. The presentation~\eqref{staple} which was considered by Stapledon~\cite{Stapledon} is unique with these properties and is called the \emph{symmetric decomposition} or \emph{Stapledon decomposition of $f$ with respect to $n$}. Let $\mathcal{AL}_n$ be the convex cone of all polynomials $f$ of degree at most $n$ such that $a_n(f)$ and $b_n(f)$ have nonnegative and unimodal coefficients. It is not hard to see that $\mathcal{AL}_n$ consists of all polynomials that satisfy the \emph{alternatingly increasing property} (see, e.g.,~\cite[Lemma 2.1]{BJM}).
\begin{thm}\label{thm:alternating}
If $f= \sum_{k=0}^n a_k t^k (1+t)^{n-k}$, where $a_k \geq 0$ for all $0\leq k \leq n$, then $\A(f) \in \mathcal{AL}_n$. In particular, $\A(f)$ is unimodal. 
\end{thm}

\begin{proof}
Since $\mathcal{I}_n (A_k)=t^{n-k-1}A_k$ for all $k\geq 1$, 
$$a_n(A_k) = A_k \cdot \frac {1-t^{n-k}} {1-t},$$
for $k\geq 1$.  Also $a_n (A_0)=1+t+\cdots + t^n$. Thus $a_n(f)$ is unimodal, since the product of two symmetric and unimodal polynomials with nonnegative coefficients is again symmetric and unimodal, see~\cite{StanleySurvey}.

By linearity we may assume $f=\prod_{i=1}^n (t+\theta_i)$, where $\theta_i \in \{0,1\}$ for all $i$. Clearly  $e_k(t+\theta_1, \ldots, t+\theta_n)$ has a nonnegative expansion in the basis $t^i(1+t)^{k-i}$.   Since 
$$
b_k\left(t^i(1+t)^{k-i}\right) = (1+t)^{k-i}\cdot \frac {1-t^{i}} {1-t},
$$
it follows that $b_k(e_k(t+\theta_1, \ldots, t+\theta_n))$ is unimodal and has nonnegative coefficients. Furthermore, 
$\mathcal{I}_m(d_m)=d_m$ and $d_m \in \mathcal{AL}_m$, see~\cite{branden2019symmetric}. Hence  
$$
b_n(e_k(t+\theta_1, \ldots, t+\theta_n)\cdot d_{n-k}(t)) =b_k(e_k(t+\theta_1, \ldots, t+\theta_n))\cdot d_{n-k}(t),
$$
has nonnegative and unimodal coefficients, and so does $\A (f)$ by \eqref{dn}.

\end{proof}

\subsection{The binomial Eulerian polynomial} 
The \emph{binomial Eulerian polynomials} may be defined by $\tilde A _n (t):=\A\left((1+t)^n\right)$, see \cite{Wachs} and the references therein. Any  polynomial $f$ for which $\mathcal{I}_n(f)=f$ may be written as 
$$
f(t)= \sum_{k=0}^{\lfloor n/2 \rfloor} \gamma_k t^k (1+t)^{n-2k},
$$
where the $\gamma_k$'s are real numbers, called \emph{$\gamma$-coefficients}. 
If $\gamma_k \geq 0$ for all $k$, then $f$ is said to be \emph{$\gamma$-positive}. It is known that  $\widetilde{A}_n$ is $\gamma$-positive~\cite{Postnikovstellohedron,Wachs}. Here we will give a different proof of this fact. Let $\sigma \in \mathfrak{S}_n^2$. An element $i \in [n]$ is a \emph{double excedance} if $\sigma(i)>i>\sigma^{-1}(i)$, and a \emph{double anti-excedance} if $\sigma(i)<i<\sigma^{-1}(i)$. Let $C$ be a cycle of $\sigma$ of length at least two, and let $x$ be the largest element of $C$. We may represent $C$ as $x_0x_1x_2\cdots x_{\ell-1}x_\ell$, where $x_0=x_\ell= x$ and  $\sigma(x_j)=x_{j+1}$ for all $0\leq j \leq \ell$. If $i=x_k$ is a double anti-excedance,  let $m$ be the smallest index such that $m > k$ and $x_m < x_k <x_{m+1}$. Let $\sigma'$ be the decorated permutation obtained by moving $x_k$ into the slot between $x_m$ and $x_{m+1}$ in the representation of $C$. Similarly, if $i=x_k$ is a double excedance, then let $m$ be the greatest  index such that $m < k$ and $x_{m-1} > x_k >x_{m}$. Let further $\sigma''$ be the decorated permutation obtained by moving $x_k$ into the slot between $x_{m-1}$ and $x_{m}$ in the representation of $C$. 
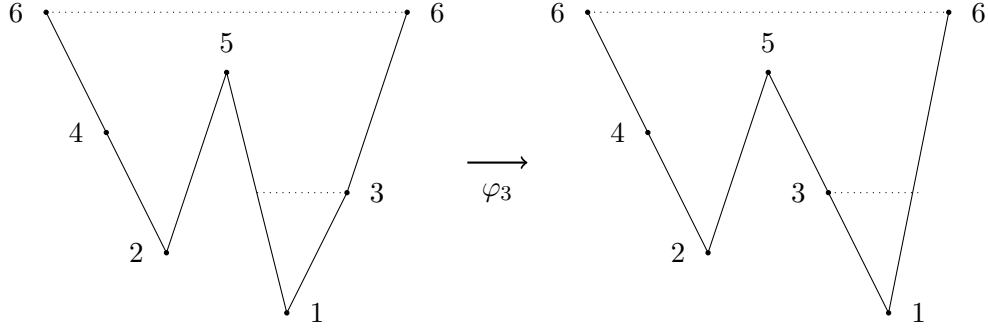
\begin{figure}[h]
\begin{tikzpicture}[scale=0.8]
\draw (0,6)--(1,4)--(2,2)--(3,5)--(4,1)--(5,3)--(6,6);
\draw[dotted] (0,6)--(6,6);
\draw[dotted] (5,3)--(3.5,3);
\filldraw[black] (0,6) circle (1pt);
\filldraw[black] (1,4) circle (1pt);
\filldraw[black] (2,2) circle (1pt);
\filldraw[black] (3,5) circle (1pt);
\filldraw[black] (4,1) circle (1pt);
\filldraw[black] (5,3) circle (1pt);
\filldraw[black] (6,6) circle (1pt);

\draw (-0.5,6) node {$6$};
\draw (0.5,4) node {$4$};
\draw (1.5,2) node {$2$};
\draw (3,5.5) node {$5$};
\draw (4.5,1) node {$1$};
\draw (5.5,3) node {$3$};
\draw (6.5,6) node {$6$};

\draw (9,6)--(10,4)--(11,2)--(12,5)--(13,3)--(14,1)--(15,6);
\draw[dotted] (9,6)--(15,6);
\draw[dotted] (13,3)--(14.5,3);
\filldraw[black] (9,6) circle (1pt);
\filldraw[black] (10,4) circle (1pt);
\filldraw[black] (11,2) circle (1pt);
\filldraw[black] (12,5) circle (1pt);
\filldraw[black] (13,3) circle (1pt);
\filldraw[black] (14,1) circle (1pt);
\filldraw[black] (15,6) circle (1pt);

\draw (8.5,6) node {$6$};
\draw (9.5,4) node {$4$};
\draw (10.5,2) node {$2$};
\draw (12,5.5) node {$5$};
\draw (12.5,3) node {$3$};
\draw (14.5,1) node {$1$};
\draw (15.5,6) node {$6$};

\draw[->,thick] (7,3.5)--(8,3.5);
\draw (7.5,3) node {$\varphi _3$};
\end{tikzpicture}
\caption{The action of $\varphi _3$ on the cycle $x_0x_1x_2x_3x_4x_5x_6=6425136$.}
\label{fig:action}
\end{figure}

For $i \in [n]$, define $\phi_i : \mathfrak{S}_n^2 \to \mathfrak{S}_n^2$ as follows. If $i$ is a fixed point of color $c$, then $\phi_i(\sigma)$ is obtained by changing $i$ to a fixed point of color $1-c$. Otherwise, if $i$ is in a cycle $C$ of length at least two, then 
$$
\phi_i(\sigma) = \begin{cases}
\sigma', &\mbox{ if } i \mbox{ is a double anti-excedance}, \\ 
\sigma'', &\mbox{ if } i \mbox{ is a double excedance}, \\
\sigma, &\mbox{ otherwise.}
\end{cases}
$$
It follows that $\phi_i$ is an involution for each $i$, and $\phi_i \phi_j = \phi_j \phi_i$ for all $i,j$. Moreover $i$ is a double excedance in $\sigma$ if and only if $i$ is a double anti-excedance in $\phi_i(\sigma)$. Hence the maps $\phi_i$, $i \in [n]$, induce an $\Z_2^n$-action $\mathfrak{S}_n^2$, see~\cite{PetterActions}. Let $\mathrm{Orb}(\sigma)$ denote the orbit of $\sigma$ under this action. 

\begin{thm}\label{gammaexp}
If $n$ is a nonnegative integer, then 
$$
\widetilde{A}_n(t)= \sum_{k=0}^{\lfloor n/2 \rfloor} \gamma_{nk} t^k (1+t)^{n-2k},
$$
where $\gamma_{nk}$ is the number of permutations in $\mathfrak{S}_n$ with exactly $k$ excedances and no double excedances. 
\end{thm}

\begin{proof}
Let $\sigma \in \mathfrak{S}_n^2$ be a decorated permutation, and let $\hat{\sigma}$ be the unique permutation in $\mathrm{Orb}(\sigma)$ with no double excedances and no fixed points of color $1$. Then 
$$
\sum_{\pi \in \mathrm{Orb}(\sigma)}t^{\exc(\sigma)}= t^{\exc(\hat{\sigma})} (1+t)^{a},
$$
where $a$ is the number of fixed points plus the number of double anti-excedances in $\hat{\sigma}$. It follows that 
$a=n-2\exc(\hat{\sigma})$, from which the theorem follows.

\end{proof}

\section{Real-rootedness and Interlacing}\label{sec:real-rooted}
If $f(t)=(t+q)^n$, then \eqref{eq:combint1} specializes to 
\[
\A\left((t+q)^n\right) = \sum_{k=0}^n \binom n k q^{n-k}A_k(t) = \sum _{\sigma \in \mathfrak{S}_n}t^{\exc (\sigma)}(t+q)^{|F(\sigma)|} \, .
\]
In this section we will prove that $\A\left((t+q)^n\right)$ is real-rooted whenever $0\leq q\leq 1$.

Let $f,g\in \mathbb{R}[t]$ be two polynomials with positive leading coefficients and real zeros only. Let further $\alpha _1\geq \alpha _2 \geq \cdots \geq \alpha _n$ be the zeros of $f$, and let $\beta_1\geq \beta _2\geq \cdots \geq \beta _m$ be the zeros of $g$. We say that $g$ \emph{interlaces} $f$, denoted by $g\preceq f$, if 
\[
\cdots \leq \beta _2 \leq \alpha _2 \leq \beta_1 \leq \alpha _1.
\]
The polynomial $g$ \emph{strictly interlaces} $f$ (written $f \prec g$) if all the non-strict inequalities between the zeros can be replaced by strict inequalities. By definition, $\deg g\leq \deg f\leq \deg g +1$. By convention, we also define $0\preceq f$ and $f\preceq 0$, where $0$ denotes the constant zero polynomial. The following lemma collects well-known basic properties of interlacing polynomials. 
\begin{lem}[{\cite[Section 3]{Wagner}}]\label{lem:basics}
Let $f,g,h\in \mathbb{R}[t]$ be polynomials with positive leading coefficients and real-zeros only. Then
\begin{itemize}
\item[(i)] if $f,g$ have only nonnegative coefficients, then $f\preceq g$ if and only if $g\preceq tf$.
\item[(ii)] $f\prec g$ if and only if $cf\prec dg$ for all $c,d>0$.
\item[(iii)] if $f\prec g$ and $f\preceq h$ then $f\prec g+h$.
\item[(iv)] if $f\prec h$ and $g\preceq h$ then $f+g \prec h$.
\end{itemize}
Furthermore, statements (i)-(iv) remain correct if $\prec$ is replaced by $\preceq$.
\end{lem}

A sequence of polynomials $(f_1,\ldots, f_n)$ is called an \emph{interlacing sequence} if $f_i\preceq f_j$ for all $1\leq i\leq j\leq n$. The following lemma shows that not all pairs of polynomials $f_i$ and $f_j$ have to be checked to guarantee an interlacing sequence.
\begin{lem}[{\cite[Lemma 2.3]{PetterPolya}}]\label{lem:PetterPolya}
Let $f_1,f_2,\ldots,f_n$ be polynomials such that $f_i\preceq f_{i+1}$ for all $1\leq i\leq n-1$, and furthermore $f_1\preceq f_n$. Then $(f_1,\ldots, f_n)$ is an interlacing sequence.
\end{lem}
Let $\mathcal F _n^+$ be the family of interlacing sequences of length $n$  of polynomials with only nonnegative coefficients. The following theorem characterizes all matrices $G=\{G_{i,j}(t)\}_{i,j}\in \mathbb{R}[t]^{m\times n}$ that preserve interlacing sequences when acting by multiplication, that is,
\begin{eqnarray*}
G&\colon & \mathcal F _n ^+ \longrightarrow \mathcal F_m^+ \\
&&(f_1,\ldots, f_n)^T \mapsto   G\cdot (f_1,\ldots, f_n)^T
\end{eqnarray*}

\begin{thm}[{\cite[Theorem 8.5]{PetterUnimodality}}]\label{thm:charinterlacingpres}
Let $G\in \mathbb{R}[t]^{m\times n}$ be a matrix of polynomials. Then $G\colon  \mathcal F _n ^+ \rightarrow \mathcal F_m^+$ if and only if 
\begin{itemize}
\item[(i)] $G_{i,j}(t)$ has nonnegative coefficients for all $1\leq i\leq m$ and all $1\leq j\leq n$; and
\item[(ii)] for all $\lambda,\mu >0$, $1\leq i<j\leq n$ and $1\leq k<\ell \leq m$
\[
(\lambda t+\mu)G_{k,j}(t)+G_{\ell, j}(t)\preceq (\lambda t+\mu) G_{k,i}(t)+G_{\ell, i}(t) \, .
\]
\end{itemize}
\end{thm}

The following theorem refines results by Haglund and Zhang~\cite{HaglundZhang} on binomial Eulerian polynomials for colored permutations. 
\begin{thm}\label{thm:main1}
For all $n\geq 0$ and $q\in (0,1]$, 
\[
\A\left((t+q)^n\right) \prec \A\left((t+q)^{n+1}\right)\, .
\]
In particular, $T\left((t+q)^n\right)$ has only real and distinct zeros.
\end{thm}
\begin{proof}
The proof uses similar techniques as in \cite{HaglundZhang}.

By a bijection due to Steingr\'imsson~\cite[Appendix]{Steingrimsson}, 
\[
\sum _{\sigma \in \mathfrak{S}_n}t^{\exc (\sigma)}(t+q)^{|F(\sigma)|}=\sum _{\sigma \in \mathfrak{S}_n}t^{\des (\sigma)}(t+q)^{\bad(\sigma)} \, ,
\]
where 
\begin{align*}
\des(\sigma) &= |\{ i \in [n-1]: \sigma(i) > \sigma(i+1)\}|, \mbox{ and } \\
\bad (\sigma) &= \left|\{k\in [n]\colon \sigma (k-1)<\sigma (k)< \sigma(m) \text{ for all } m>k\}\right|.
\end{align*}

We consider the polynomials
\[
f_{n,i} =\sum _{\sigma \in \mathfrak{S}_n\atop \sigma (1)=i}t^{\des (\sigma)}(t+q)^{\bad(\sigma)} \, .
\]
By observing  the change of $\des (\sigma)$ and $\bad(\sigma)$ while inserting a letter $i$ in first position of a permutation on $n$ letters in one line notation (after shifting all letters $j\geq i$ to $j+1$), we obtain the recursions
\begin{eqnarray*}
f_{n+1,1} &=& (t+q)\sum _{k=1}^n f_{n,k}, \\
f_{n+1,i} & =& \frac{t}{t+q}f_{n,1}+\sum_{k=2}^{i-1}tf_{n,k}+\sum_{k=i}^{n}f_{n,k}, \text{ for }2\leq i\leq n+1 \, ,
\end{eqnarray*}
Defining $p_{n,1} =\frac{f_{n,1}}{t+q}$ and $p_{n,i} =f_{n,i}$ for $i>1$, we obtain
\begin{eqnarray}
p_{n+1,1} &=& (t+q)p_{n,1}+\sum _{k=2}^n p_{n,k}, \ \ \ \mbox{ and } \label{eqn:recursion1}\\
p_{n+1,i} & =& \sum_{k=1}^{i-1}tp_{n,k}+\sum_{k=i}^{n}p_{n,k}, \text{ for }2\leq i\leq n+1 \label{eqn:recursion2}\, .
\end{eqnarray}
We claim that $(p_{n,1},\ldots,p_{n,n})$ is an interlacing sequence for all $n\geq 1$. If $n=1$, then the claim holds trivially since $f_{1,1}=t+q$ and $p_{1,1}=1$.

For $n\geq 1$, we observe that \eqref{eqn:recursion1} and \eqref{eqn:recursion2} are equivalent to
\[
\begin{bmatrix}
p_{n+1,1}^{q}\\
p_{n+1,2}^{q}\\
\vdots\\
p_{n+1,n+1}^{q}\\
\end{bmatrix}
=
G
\begin{bmatrix}
p_{n,1}^{q}\\
p_{n,2}^{q}\\
\vdots\\
p_{n,n}^{q}\\
\end{bmatrix}
\]
where
\[
G=
\begin{bmatrix}
t+q &1&\cdots &1\\
t&1&\cdots &1\\
t&t&\cdots &1\\
\vdots&&\ddots&\\
t&t&\cdots &t
\end{bmatrix}
\, .
\]
To prove the claim it suffices to show that $G$ preserves interlacing of sequences of polynomials with nonnegative coefficients. By Theorem~\ref{thm:charinterlacingpres}, it suffices to check this for all $2\times 2$ submatrices of $G$:
\[
\begin{bmatrix}
t&t\\
t&t
\end{bmatrix}
,
\begin{bmatrix}
t&1\\
t&t
\end{bmatrix}
,
\begin{bmatrix}
1&1\\
t&t
\end{bmatrix}
,
\begin{bmatrix}
1&1\\
t&1
\end{bmatrix}
,
\begin{bmatrix}
t&1\\
t&1
\end{bmatrix}
,
\begin{bmatrix}
1&1\\
1&1
\end{bmatrix}
,
\begin{bmatrix}
t+q&1\\
t&t
\end{bmatrix}
,
\begin{bmatrix}
t+q&1\\
t&1
\end{bmatrix}
\, .
\]
The first six matrices have been shown to preserve interlacing sequences in~\cite[Corollary 8.7]{PetterUnimodality}. The last two matrices can be factored into matrices that preserve interlacing sequences. Indeed,
\begin{eqnarray*}
\begin{bmatrix}
t+q &1\\
t&t
\end{bmatrix}
&=&
\begin{bmatrix}
q &1\\
t&0
\end{bmatrix}
\begin{bmatrix}
1 &1\\
t&1-q
\end{bmatrix}
 \, ,\\
\begin{bmatrix}
t+q &1\\
t&1
\end{bmatrix}
&=&
\begin{bmatrix}
1 &1\\
0&1
\end{bmatrix}
\begin{bmatrix}
q &0\\
t&1
\end{bmatrix}
\, .
\end{eqnarray*}
Since $0\leq q\leq 1$, all factors have only nonnegative coefficients. To see that all factors preserve interlacing we furthermore need to check the second condition in Theorem~\ref{thm:charinterlacingpres}.  For example, for the first factor in the first product we need to check that for all $\mu,\lambda >0$
\[
(\lambda t+\mu)\cdot 1 + 0 \preceq (t\lambda +\mu )\cdot q+t \, .
\]
This is equivalent to 
\[
-\frac{\mu}{\lambda} \leq -\frac{\mu q}{\lambda q +1}
\]
which is true. The other three factors preserve interlacing sequences by a similar argument.

In summary, all factors and thus also their products preserve interlacing. Therefore, $(p_{n,1},\ldots,p_{n,n})$ is an interlacing sequence. 

We furthermore claim that $p_{n,1}$ strictly interlaces $p_{n,2}$ for all $n\geq 2$ whenever $q>0$. We again argue by induction. For $n=2$ we have $p_{2,1}=t+q$ and $p_{2,2}=t$ and thus the claim holds. For $n\geq 2$. we obtain
\[
p_{n+1,2}=tp_{n,1}+\sum _{i=2}^n p_{n,i} \succ tp_{n,1}+\sum _{i=2}^n p_{n,i} +qp_{n,1} = p_{n+1,1}
\]
by Lemma~\ref{lem:basics}, assuming that $p_{n,1}\prec p_{n,2}$.

To complete the proof, we observe that 
\begin{align*}
\A\left((t+q)^{n+1}\right) &=\sum _{i=1}^{n+1} f_{n+1,i}=(t+q)p_{n+1,1}+\sum _{i=2}^{n+1} p_{n+1,i}  \\
&\succeq p_{n+1,1}=(t+q)p_{n,1}+\sum _{i=2}^n p_{n,i} \\
&=\A\left((t+q)^{n}\right),
\end{align*}
since $p_{n+1,1}$ interlaces $tp_{n+1,1}$ as well as $p_{n+1,i}$ for $i>1$. Since $q>0$ the polynomial $p_{n+1,1}$ strictly interlaces $p_{n+1,2}$, and thus $\succeq$ can be replaced with $\succ$ by Lemma~\ref{lem:basics}. This completes the proof. \qedhere

\end{proof}

\begin{thm}\label{thm:interlacing}
For all $0\leq p<q\leq 1$, 
\begin{equation}
\A ((t+q)^n)\prec \A ((t+p)^n) 
\end{equation}
and
\begin{equation}\label{eq:reverse}
\mathcal I _n (\A ((t+p)^n) \preceq \A ((t+p)^n).
\end{equation}
If $p\not \in \{0,1\}$ then $\preceq$ can be replaced by $\prec$ in Equation~\eqref{eq:reverse}.
\end{thm}
\begin{proof}
Let $0<p \leq 1$. For all $\epsilon >0$,
\[
\A ((t+p+\epsilon)^n)=\A ((t+p)^n)+\epsilon n \A ((t+p)^{n-1})+\epsilon ^2 r(t),
\]
where $r(t)$ is a polynomial of degree $n-2$. Since $\A ((t+p)^{n-1})$ strictly interlaces $\A ((t+p)^n)$ by Theorem~\ref{thm:main1}, 
\[
n \A ((t+p)^{n-1})+\epsilon  r(t) \prec \A ((t+p)^n)
\]
for sufficiently small $\epsilon >0$, and thus by Lemma~\ref{lem:basics},
\[
\A ((t+p+\epsilon)^n) \prec \A ((t+p)^n),
\]
for sufficiently small $\epsilon >0$. 
This is also true for $p=0$, since $A_{n-1}/t \prec A_n/t$.

Similarly, if $0\leq p \leq 1$, then $\A ((t+p)^n) \prec \A ((t+p-\epsilon)^n)$ for all $\epsilon >0$ sufficiently small. Let $0 \leq p < q \leq 1$. By compactness there exists a finite sequence $1 = p_0 > p_1 > \cdots > p_{k-1} >p_k=0$, such that $p=p_i$ and $q=q_j$ for some $i,j$, and $\A ((t+p_{\ell-1})^n) \prec \A ((t+p_{\ell})^n)$ for all $1\leq \ell \leq k$. Let $g_\ell =  \A ((t+p_{\ell})^n)$. Then 
$$
A_n/t=\mathcal{I}_n(g_{k})   \prec  \cdots \prec \mathcal{I}_n(g_1) \prec \mathcal{I}_n(g_0)=g_0 \prec g_1 \prec \cdots \prec g_k= A_n, 
$$
and so the theorem follows from Lemma~\ref{lem:PetterPolya}.  \qedhere

\end{proof}
Let $a_n(t)=t^na(1/t)$ and $b_n(t)=t^{n-1}b(1/t)$ be the unique polynomials such that $f(t)=a_n(t)+tb_n(t)$. Then $f$ is said to have an \emph{interlacing symmetric decomposition} if $b_n(t)$ interlaces $a_n(t)$. A systematic study of interlacing symmetric decompositions was initiated in~\cite{branden2019symmetric}. The following is an immediate consequence of \cite[Theorem 2.7]{branden2019symmetric}, Theorem~\ref{thm:alternating} and equation~\eqref{eq:reverse} in Theorem~\ref{thm:interlacing}.
\begin{cor}\label{cor:interlacingsym}
The polynomial $\A ((t+q)^n)$ has an interlacing symmetric decomposition for all $0\leq q\leq 1$.
\end{cor}

\section{Topological interpretation}\label{sec:topolint}

A \emph{simplicial complex} on the groundset $[n]$ is a non-empty set $\Delta$ of subsets of $[n]$ that is closed under taking subsets. The elements of $\Delta$ are called \emph{faces}. The (combinatorial) dimension\footnote{Notice that the combinatorial dimension differs from the usual topological dimension by one.}, $\dim \sigma$, of a face $\sigma$ is defined to be $\dim \sigma =| \sigma |$. The  dimension of the simplicial complex $\Delta$ is defined as the maximal dimension of one of its faces. The number of faces of dimension $i$ is denoted by $f_i (\Delta)$. In particular, $f_0(\Delta)=1$. The \emph{$f$-polynomial} records the number of faces of $\Delta$ according to their dimension. If $\Delta$ is a $d$-dimensional simplicial complex then 
\[
f_\Delta (t) \ = \ \sum _{\sigma \in \Delta}t^{\dim \sigma} \ = \ \sum _{i=0}^d f_i (\Delta)t^i \, .
\]
The \emph{$h$-polynomial} $h_\Delta (t)$ of a $d$-dimensional simplicial complex $\Delta$ is defined by
\[
h_\Delta (t) = (1-t)^d f_\Delta \left( \frac{t}{1-t}\right) \, .
\]
To any simplicial complex on the groundset $[n]$ we can associate a simplicial complex $\Delta '$ on the ground set $2^{[n]}\setminus \emptyset \sqcup [n']$ where $2^{[n]} \setminus \emptyset$ denotes the set of all non-empty subsets of $[n]$, and $[n']=\{1',2',\ldots, n'\}$ an identical copy of $[n]$. The faces of $\Delta '$ consist of all sets of the form
\[
\{ F_1 \subsetneq F_2 \subsetneq \cdots \subsetneq F_k  \colon 0\leq k\leq n\} \sqcup \{i' : i \in S\}, \quad \text{ where } \quad S\cap F_k=\emptyset, 
\]
and where $F_1 \subsetneq F_2 \subsetneq \cdots \subsetneq F_k$ is a chain of non-empty elements of $\Delta$, called \emph{flags}. It is not hard to see that if $\Delta$ is a simplicial complex on the groundset $[n]$ then $\Delta '$ is an $n$-dimensional simplicial complex.
The following result establishes an interpretation of $\A (f)$ as the $h$-polynomial of a simplicial complex.

\begin{thm}\label{thm:topoint}
Let $\Delta$ be a simplicial complex with groundset $[n]$ and let $f_\Delta (t)=f_0+f_1t+\cdots +f_dt^d$ be its $f$-polynomial. Then the $h$-polynomial of $\Delta '$ is equal to
\[
h_{\Delta '} (t) = \sum _{i=0}^d f_i A_i (t)=\A (f) \, .
\]
\end{thm}
\begin{proof}
We first determine the $f$-polynomial of $\Delta '$. We observe that

\begin{eqnarray*}
f_{\Delta '}(t)&=& \sum _{F\in \Delta} \sum _{\{F_1 \subsetneq F_2\subsetneq \cdots \subsetneq F_k=F \}}\sum _{S\colon S\cap F=\emptyset} t^{k+|S|}\\
&=& \sum _{F\in \Delta} (t+1)^{n-\dim F}\sum _{\{F_1 \subsetneq F_2\subsetneq \cdots \subsetneq F_k=F \}}t^{k} \, .
\end{eqnarray*}
Further, if $\dim F=\ell$, then the set of flags $\{F_1 \subsetneq F_2\subsetneq \cdots \subsetneq F_k=F \} $ of length $k$ with maximal element $F$ is in bijection with the collection of  surjective maps from $[\ell]$ to $[k]$. Their number is equal to $k!S(\ell, k)$ where $S(\ell, k)$ is the Stirling numbers of the second kind. Setting $S_\ell (t)=\sum _{k=1}^\ell k! S(\ell, k)t ^k$ we obtain  
\[
f_{\Delta '}(t)=\sum _{\ell =0}^d \sum _{F \in \Delta \atop \dim F=\ell} (1+t)^{n-\ell}\sum _{k=1}^\ell k! S(\ell, k)t^k=\sum _{\ell=0}^d f_\ell (1+t)^{n-\ell}S_\ell (t) \, .
\]
The Stirling numbers and the Eulerian numbers are related via the following classical identity, see, e.g.,~\cite[p. 314]{WagnerChains}:
\[
A_\ell (t)= (1-t)^\ell S_\ell \left( \frac{t}{1-t}\right) \, .
\]
We therefore obtain
\begin{eqnarray*}
h_{\Delta '} (t) &=&(1-t)^n f _{\Delta '} \left(\frac{t}{1-t}\right)\\
&=&(1-t)^n \sum _{\ell=0}^d f_\ell \left(\frac{1}{1-t}\right)^{n-\ell}S_\ell \left(\frac{t}{1-t}\right)\\
&=&\sum _{\ell=0}^d f_\ell (1-t)^{\ell}S_\ell \left(\frac{t}{1-t}\right)\\
&=&\sum _{\ell=0}^d f_\ell A_\ell (t),\\
\end{eqnarray*}
as desired.
\end{proof}
As a corollary we obtain the following result for simplicial complexes with nonnegative $h$-polynomials.
\begin{cor}
Suppose $\Delta$ is a $d$-dimensional simplicial complex with $h$-polynomial $h_\Delta =h_0+h_1t+\cdots +h_dt^d$, where $h_0,h_1,\ldots, h_d\geq 0$. Then the $h$-polynomial of $\Delta '$ is alternatingly increasing.
\end{cor}
\begin{proof}
By definition,
\[
f_\Delta (t)=\sum _{i=0}^d h_i t^i (1+t)^{d-i} \, .
\]
The claim follows therefore directly from Theorems~\ref{thm:alternating} and~\ref{thm:topoint}.
\end{proof}

\section{Ehrhart theory}\label{Ehrhart}
A lattice polytope is the convex hull of finitely many points in $\mathbb{Z}^d$. Ehrhart theory is concerned with the enumeration of lattice points in integer dilates of lattice polytopes. Ehrhart~\cite{Ehrhart} showed that for every lattice polytope $P\subset \mathbb{R}^d$ the number of lattice points in the $n$-th dilate of $P$, $|nP\cap \mathbb{Z}^d| $, agrees with a polynomial $\Ehr _P (n)$ of degree $\dim P$ for all integers $n\geq 0$. The polynomial $\Ehr _P (n)$ is called the \emph{Ehrhart polynomial} of $P$. If $P$ is a $d$-dimensional lattice polytope, then the \emph{$h^\ast$-vector} $h^\ast (P)=(h_0(P),\ldots,h_d(P))$ is defined by
\[
\Ehr \nolimits _P (n)=h_0(P){n+d\choose d}+h_1(P){n+d-1\choose d}+\cdots + h_d(P){n\choose d} \, .
\]
The polynomial $h^\ast (P)=h_0(P)+h_1(P)t+\cdots +h_d(P)t^d$ is called the \emph{$h^\ast$-polynomial} of $P$. In this section we give  an interpretation of $\A((\theta _1t+1)\cdots (\theta _d t+1))$ as the $h^\ast$-polynomial of a lattice polytope, whenever $\theta _1,\ldots, \theta _d$ are positive integers.

To that end, we will make use of half-open decompositions of polytopes~\cite{CombinatorialPositivity}. Consider a $d$-dimensional polytope $P \subset \mathbb{R} ^d$ with facets $F_1,\ldots, F_m$. A point $q\in \mathbb{R}^d$ is in \emph{general position} if $q$ is not contained in the affine hull of $F_i$, for any $1 \leq i \leq m$. A facet $F_i$ is \emph{visible} from $q$, if $P$ and $q$ lie on opposite sides of the affine hull of $F_i$. Let $I_q (P)=\{i\in [m]\colon F_i \text{ is visible from } q\}$. Then the set
\[
H_q (P)=P \setminus \bigcup _{i \in I_q (P)}F_i
\]
is a \emph{half-open polytope}. The notions of Ehrhart polynomial and $h^\ast$-polynomial can be extended to half-open polytopes in a natural way by means of the inclusion-exclusion principle.
\begin{lem}[{\cite{CombinatorialPositivity}}]\label{lem:halfopendecomp}
Let $P_1,\ldots,P_k$ be the maximal cells of a polytopal subdivision of $P=P_1\cup P_2\cup \cdots \cup P_k$. Let $q\in \relint P$ be a point that is in general position with respect to each $P_i$. Then 
\[
h^\ast _P (t) = \sum _{i=1}^k h^\ast _{H_q (P_i)} (t) \, .
\]
\end{lem}
We will also need the following lemma which is folklore (see, e.g.,\cite{BeckRobins}).
\begin{lem}\label{lem:folklore}
Let $P\in \mathbb{R}^d$ be a lattice polytope and $P\star \mathbf{0}:=\conv ((P\times \{1\}) \cup \mathbf{0})\subset \mathbb{R}^{d+1}$. Then
\[
h^\ast _P (t) = h^\ast _{P\star \mathbf{0}} (t) \, .
\]
\end{lem}
The preceding lemma remains true for half-open polytopes if we set
\[
H_q (P) \star  \mathbf{0} := \left(P\star \mathbf{0}\right) \setminus \bigcup _{i \in I_q (P)}\left(F_i\star \mathbf{0} \right)
\]
with the notation above.

For positive integers $\theta _1,\theta _2,\ldots, \theta _d$, let  $P_d(\theta _1,\ldots, \theta _d)\subset \mathbb{R}^d$ be the $d$-dimensional lattice polytope defined by
\[
P_d(\theta _1,\ldots, \theta _d) = \conv \left( (-\Delta _d)\cup  Q_d(\theta _1,\ldots, \theta _d)\right) \, ,
\]
where $\Delta _d =\conv (\mathbf{0}, \mathbf{e}_1,\ldots,\mathbf{e}_d)$ is the $d$-dimensional standard simplex and
\[
Q_d(\theta _1,\ldots, \theta _d)=\{\mathbf{x}\in \mathbb{R}^d\colon 0\leq x_i \leq \theta _i \text{ for all } 1\leq i\leq d \}
\]
is an axis-parallel parallelepiped with side lengths $\theta _1,\ldots, \theta _d$.

\begin{thm}\label{thm:hstarinterpretation}
For positive integers $\theta _1,\theta _2,\ldots, \theta _d$,
\[
h^\ast _{P_d(\theta _1,\ldots, \theta _d)}(t)\ = \ \A\left(\prod _{i=1}^d (\theta _i t+1) \right) \, .
\]
\end{thm}
\begin{proof}
For all $\sigma \in \{-1,1\}^d$, let $\sigma \mathbb{R}_{\geq 0}=\{(x_1\sigma _1,\ldots, x_d\sigma _d)\colon x_i\geq 0 \text{ for all } i\}$ and let $\sigma^+=\{i\in [d]\colon \sigma _i=1\}$ and $\sigma^-=[d]\setminus \sigma ^ +$.

Both polytopes $\Delta _d$ and $Q_d(\theta _1,\ldots, \theta _d)$ are anti-blocking polytopes, that is, if a point is contained in the polytope, then so are all their projections on coordinate subspaces. Thus, by~\cite[Lemma 2.4]{artstein2020geometric}, the subdivision induced by the coordinate hyperplanes yields the decomposition
\begin{eqnarray*}
P_d(\theta _1,\ldots, \theta _d) &=&\bigcup _{\sigma \in \{-1,1\}^d} \conv \left(\pi _{\sigma^-}(-\Delta _d)\cup \pi _{\sigma^+}(P_d(\theta _1,\ldots, \theta _d)) \right) \, ,
\end{eqnarray*}
where $\pi _I$ denotes the projection onto the coordinate subspace $\mathbb{R}^I=\{\mathbf{x}\in \mathbb{R}^d\colon x_i=0 \text{ for } i \not \in I\}$. Now, choosing a generic point $q\in P_d(\theta _1,\ldots, \theta _d) \cap \mathbb{R}_{<0}^d$ results in the half-open decomposition 
\begin{eqnarray*}
P_d(\theta _1,\ldots, \theta _d) &=&\bigsqcup _{\sigma \in \{-1,1\}^d} H_q \left(\conv \left((-\Delta _{\sigma^-})\cup Q_{\sigma^+}(\mathbf{\theta}) \right)\right) \, ,
\end{eqnarray*}
where $\Delta _I=\conv\left(\{\mathbf{e}_i\colon i\not \in I\}\cup\mathbf{0}\right) \subset \mathbb{R}^I$ denotes a standard simplex and
\[
Q_{I}(\mathbf{\theta})=\{\mathbf{x}\in \mathbb{R}^I \colon 0\leq x_i\leq \theta _i \text{ for all } i\not \in I\}
\]
a parallelepiped for all $I\subseteq [d]$. 
\begin{figure}[h]
\begin{tikzpicture}[scale=1.6]

\draw[thick] (0.1,1)--(2,1)--(2,0.1);
\draw[thick] (2,0)--(0.1,-0.95);
\draw[thick] (0,-1)--(-1,0);
\draw[thick] (-0.9,0.1)--(0,1);

\draw[dotted,thick] (0.1,1)--(0.1,0.1)--(2,0.1);
\draw[dotted,thick] (0.1,-0.95)--(0.1,0);
\draw[dotted,thick] (-0.9,0.1)--(0,0.1);

\fill[fill=gray!15] (0.1,1)--(2,1)--(2,0.1)--(0.1,0.1)--(0.1,1);
\fill[fill=gray!15] (2,0)--(0.1,-0.95)--(0.1,0)--(2,0);
\fill[fill=gray!15] (0,-1)--(-1,0)--(0,0);
\fill[fill=gray!15] (-0.9,0.1)--(0,1)--(0,0.1);
\draw[->] (0,-1.5)--(0,1.5);
\draw[->] (-1.5,0) -- (2.5,0);

\draw (2.5,-0.25) node {$x_1$};
\draw (-0.25,1.5) node {$x_2$};

\draw (2,-0.25) node {$2$};
\draw (-0.25,1) node {$1$};
\draw (-1.05,0.25) node {$-1$};
\draw (0.25,-1.05) node {$-1$};
\end{tikzpicture}
\caption{The lattice polytope $P_2(2,1)$ and its half-open decomposition.}
\end{figure}
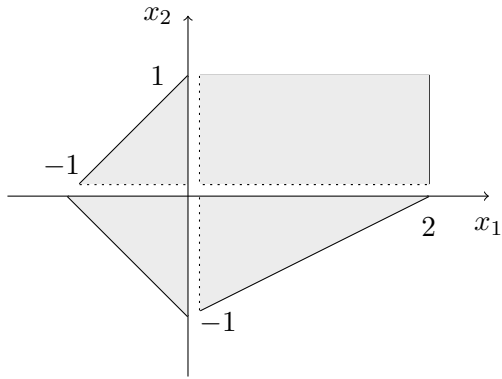

We observe that for all $\sigma \in \{\pm 1\}^d$
\[
H_q \left(\conv \left((-\Delta _{\sigma^-})\cup Q_{\sigma^+}(\mathbf{\theta}) \right)\right) \cong R_{\sigma^+}(\mathbf{\theta})\star \mathbf{0}^{\ast |\sigma ^-|},
\]
where $R_I (\theta)$ denotes the half-open parallelepiped $R_I (\theta)=\{\mathbf{x}\in \mathbb{R}^I \colon 0< x_i\leq \theta _i \text{ for all } i \in I\}$, and $\mathbf{0}^{\ast a}$ denotes applying $a$ many times the operation $\star\mathbf{0}$. The $h^\ast$-polynomial of the half-open parallelepiped $R_I (\theta)$ is equal to $h^\ast _{R_I (\theta)}=\left(\prod _{i\in I}\theta _i\right)A_{|I|}(t)$ (see, e.g.,~\cite[Theorem 4.7]{BJM}). Therefore, by Lemma~\ref{lem:halfopendecomp} and Lemma~\ref{lem:folklore},
\[
h^\ast _{P_d(\theta _1,\ldots, \theta _d)} (t) =\sum _{\sigma \in \{-1,1\}^d} \left(\prod _{i\in \sigma ^+} \theta _i\right) A_{|\sigma ^+|}(t)=T\left(\prod _{i=1}^d (\theta _i t+1) \right) \, ,
\]
as desired.
\end{proof}
The following corollary is an immediate consequence of Theorems~\ref{thm:hstarinterpretation} and~\ref{thm:alternating}.
\begin{cor}\label{cor:unimodalhstar}
The coefficients of the $h^\ast$-polynomial of $P_d(\theta _1,\ldots, \theta _d)$ form a unimodal sequence for all positive integers $\theta _1,\ldots, \theta _d$.
\end{cor}
Lattice polytopes with unimodal $h^\ast$-polynomials are of particular interest in Ehrhart theory~\cite{braun2016unimodality}. A question by Schepers and Van Langenhoven~\cite{Schepers} asks whether every polytope having the integer decomposition property (IDP) has a unimodal $h^\ast$-polynomial. This question is part of a hierarchy of conjectures and questions that originate from a conjecture by Stanley~\cite{StanleySurvey}, originally formulated in the language of commutative algebra, which is often referred to as \emph{Stanley's Unimodality Conjecture}. A polytope $P\subset \mathbb{R}^d$ has the \emph{integer decomposition property (IDP)} if for all integers $n$ and all lattice points $p\in nP\cap \mathbb{Z}^d$ in the $n$-th dilate of $P$ there are lattice points $p_1,\ldots, p_n\in P\cap \mathbb{Z}^d$ such that $p=p_1+\cdots +p_n$. It is not hard to see that lattice parallelepipeds are IDP polytopes. Moreover, the operation $\star \mathbf{0}$ as well as taking unions preserves the integer decomposition property. From the decomposition considered in the proof of Theorem~\ref{thm:hstarinterpretation} it therefore follows that $P_d(\theta _1,\ldots, \theta _d)$ is an IDP polytope for all positive integers $\theta _1,\ldots, \theta _n$. Corollary~\ref{cor:unimodalhstar} therefore provides further evidence for the aforementioned conjectures.\\

\bigskip

\noindent {\bf Acknowledgements.} 
We would like to thank Francesco Brenti for several interesting and fruitful discussions. We would also like to thank the anonymous referee for helpful comments. This project originated at the AIM workshop ``Polyhedral geometry and partition theory'' in 2016. We are thankful to the organizers Federico Ardila, Benjamin Braun, Peter Paule and Carla Savage and the American Institute of Mathematics (AIM) for organizing and hosting the workshop. We also would like to thank Tewodros Amdeberhan, Katie Gedeon, Apoorva Khare, Kyle Petersen, Carla Savage and Mirko Visontai for interesting discussions during the workshop.

PB is a Wallenberg Academy Fellow supported by the Knut and Alice Wallenberg foundation and the G\"oran Gustafsson foundation. 
KJ is supported by the Wallenberg AI, Autonomous Systems and Software Program funded by the Knut and Alice Wallenberg Foundation, as well as Swedish Research Council grant 2018-03968 and the G\"oran Gustafsson foundation.

\bibliographystyle{siam}
\bibliography{Euler}

\end{document}